\documentclass[10pt,rqno,letterpaper]{article}
\usepackage[utf8]{inputenc}
\usepackage{amsmath}
\usepackage{amsfonts}
\usepackage{amssymb}
\usepackage{graphicx}
\usepackage{mathrsfs}
\usepackage{upref,amsthm,amsxtra,exscale}
\usepackage{cite}
\usepackage[colorlinks=true,urlcolor=blue,
citecolor=red,linkcolor=blue,linktocpage,pdfpagelabels,
bookmarksnumbered,bookmarksopen]{hyperref}
\usepackage{fullpage}

\allowdisplaybreaks %Allows page breaks inside the align environment 
\usepackage{hyperref}

\usepackage{subcaption}
\usepackage{caption}

\newtheorem{theorem}{Theorem}[section]

\newtheorem{lemma}[theorem]{Lemma}

\numberwithin{equation}{section}

\usepackage{xcolor}

\def\cI{\mathcal{I}}

\def\({\left(}
\def\){\right)}

\def\r{\mathbb{R}}
\def\rm{\mathbb{R}^m}
\def\sm{\mathbb{S}^m}

\def\n{\mathbb{N}}

\def\dm{\mathrm{\,d}\mu_g}

\def\rh{\rightharpoonup}
\def\im{\int_{M}}

\def\o{\Omega}

\def\cC{\mathcal{C}}
\def\cE{\mathcal{E}}
\def\cH{\mathcal{H}}
\def\cK{\mathcal{K}}
\def\cM{\mathcal{M}}
\def\cN{\mathcal{N}}
\def\cP{\mathcal{P}}
\def\cJ{\mathcal{J}}
\def\cT{\mathcal{T}}
\def\cU{\mathcal{U}}
\def\bf{\mathbf}

\author{M\'onica Clapp\footnote{M. Clapp was partially supported by CONACYT (Mexico) through the grant for project A1-S-10457.} \ and \ Angela Pistoia\footnote{A. Pistoia was partially supported by Fondi di Ateneo ``Sapienza'' Universit\`a di Roma (Italy). }}
\title{Yamabe systems and optimal partitions on manifolds with symmetries}
\date{\today}

\begin{document}
\maketitle

\centerline{\em Dedicated to Norman Dancer on the occasion of his 75th birthday}

\begin{abstract}
We prove the existence of regular optimal $G$-invariant partitions, with an arbitrary number $\ell\geq 2$ of components, for the Yamabe equation on a closed Riemannian manifold $(M,g)$ when $G$ is a compact group of isometries of $M$ with infinite orbits. To this aim, we study a weakly coupled competitive elliptic system of $\ell$ equations, related to the Yamabe equation. We show that this system has a least energy $G$-invariant solution with nontrivial components  and we show that the limit profiles of the its components   separate spatially as the competition parameter goes to $-\infty$, giving rise to an optimal partition. 
 For $\ell=2$ the optimal partition obtained yields a least energy sign-changing $G$-invariant solution to the Yamabe equation with precisely two nodal domains. 
\medskip

\medskip

\end{abstract}

\section{Introduction}

This paper is concerned with the existence and asymptotic behavior of solutions to the weakly coupled competitive Yamabe system
\begin{equation} \label{eq:s}
\begin{cases}
\mathscr{L}_g u_i:=-\Delta_g u_i + \kappa_mS_gu_i = |u_i|^{2^*-2}u_i + \sum\limits_{\substack{j=1 \\ j\neq i}}^\ell\lambda_{ij}|u_j|^\frac{2^*}{2}|u_i|^\frac{2^*}{2}u_i\quad\text{on }M, \\
u_i\neq 0,\qquad i=1,\ldots,\ell,
\end{cases}
\end{equation}
where $(M,g)$ is a closed Riemannian manifold of dimension $m\geq 3$, $S_g$ is its scalar curvature, $\Delta_g:=\mathrm{div}_g\nabla_g$ is the Laplace-Beltrami operator, $\kappa_m:=\frac{m-2}{4(m-1)}$, $2^*:=\frac{2m}{m-2}$ is the critical Sobolev exponent, and $\lambda_{ij}=\lambda_{ji}<0$. We assume that the quadratic form induced by the conformal Laplacian $\mathscr{L}_g$ is coercive. 

This system was recently studied by Clapp, Pistoia and Tavares \cite{cpt}. Here we complement the results obtained in \cite{cpt} by considering manifolds with symmetries. 

Let $G$ be a compact group of isometries of $M$ and let $Gp:=\{\gamma p:\gamma\in G\}$ denote the $G$-orbit of a point $p\in M$. Recall that a subset $X$ of $M$ is said to be $G$-invariant if $Gp\subset X$ for every $p\in\o$ and a function $u:X\to\r$ is called $G$-invariant if it is constant on every $G$-orbit of $\o$. We shall say that a solution $(u_1,\ldots,u_\ell)$ to the system \eqref{eq:s} is $G$-invariant if every component $u_i$ is $G$-invariant.

We prove the following result.

\begin{theorem}\label{thm:existence}
If $1\leq\dim (Gp)<m$ for every $p\in M$, then the system \eqref{eq:s} has a least energy $G$-invariant solution and infinitely many $G$-invariant solutions.
\end{theorem}

To describe the limit profile of least energy $G$-invariant solutions to the system \eqref{eq:s} as $\lambda_{ij}\to-\infty$, we consider the Dirichlet problem
\begin{equation} \label{eq:u}
\begin{cases}
\mathscr{L}_gu = |u|^{2^*-2}u &\text{ in }\Omega,\\
u=0 &\text{ on }\partial \Omega,\\
u\text{ is }G\text{-invariant},
\end{cases}
\end{equation}
in a $G$-invariant open subset $\o$ of $M$, and set
\begin{equation} \label{eq:cOmega}
c_\Omega^G:=\inf\Big\{\frac{1}{m}\im|u|^{2^*}\dm:u\neq 0, \ u\text{ solves }\eqref{eq:u}\Big\}.
\end{equation}
Let $\cP_\ell^G:=\{\{\Omega_1,\ldots,\Omega_\ell\}:\,\Omega_i\neq\emptyset \text{ is a }G\text{-invariant open subset of }M\text{ and }\Omega_i\cap \Omega_j=\emptyset\text{ if }i\neq j \}$. We shall say that $\{\Omega_1,\ldots,\Omega_\ell\}\in\cP_\ell^G$ is \emph{an optimal $(G,\ell)$-partition} for the Yamabe equation
\begin{equation} \label{eq:y}
\mathscr{L}_g u = |u|^{2^*-2}u\qquad\text{on }M,
\end{equation}
if
$$\sum_{i=1}^\ell c_{\Omega_i}^G=\inf_{\{\Theta_1,\ldots,\Theta_\ell\}\in\cP_\ell^G}\;\sum_{i=1}^\ell c_{\Theta_i}^G.$$

The relation between variational elliptic systems having large competitive interaction and optimal partition problems was observed by Conti, Terracini and Verzini in \cite{ctv1,ctv2} and Chang, Lin, Lin and Lin in \cite{clll}, and has been extensively studied since. An ample list of references is given in \cite{cpt}. Our next result describes this relation for the system \eqref{eq:s}.
\begin{theorem} \label{thm:op}
Assume $1\leq\dim (Gp)<m$ for every $p\in M$. Let $\lambda_{n}<0$ be such that $\lambda_{n}\to -\infty$. For each $n\in\n$, let $(u_{n,1},\ldots,u_{n,\ell})$ be a least energy $G$-invariant solution to the system \eqref{eq:s} with $\lambda_{ij}=\lambda_n$ for all $i\neq j$, such that $u_{n,i}>0$ for all $n\in\n$ and $i=1,\ldots\ell$. Then, after passing to a subsequence, we have that
\begin{itemize}
\item[$(i)$]$u_{n,i}\to u_{\infty,i}$ strongly in $H_g^1(M)\cap\cC^{0,\alpha}(M)$ for every $\alpha\in (0,1)$, where  $u_{\infty,i}\geq 0$,\, $u_{\infty,i}\neq 0$,\, and $u_{\infty,i}|_{\Omega_i}$ is a least energy $G$-invariant solution to the problem \eqref{eq:u} in $\Omega_i:=\{p\in M:u_{\infty,i}(p)>0\}$ for each $i=1,\ldots,\ell$. Moreover,
\[
\int_M \lambda_n u_{n,i}^\frac{2^*}{2} u_{n,j}^\frac{2^*}{2}\to 0 \text{ \ as \ } n\to \infty\quad \text{whenever } i\neq j.
\]
\item[$(ii)$] $u_{\infty,i}$ is $G$-invariant and $u_{\infty,i}\in \cC^{0,1}(M)$ for each $i=1,\ldots, \ell$.
\item[$(iii)$]$\{\Omega_1,\ldots,\Omega_\ell\}\in\cP_\ell^G$ and it is an optimal $(G,\ell)$-partition for the Yamabe equation on $(M,g)$. In particular, each $\Omega_i$ is connected.
\item[$(iv)$] $M\smallsetminus\bigcup_{i=1}^\ell\Omega_i=\mathscr R\cup\mathscr S$, where $\mathscr R\cap\mathscr S=\emptyset$, $\mathscr R$ is an  $(m-1)$-dimensional $\cC^{1,\alpha}$-submanifold of $M$ and $\mathscr S$ is a closed subset of $M$ with Hausdorff measure $\leq m-2$. In particular, $M=\cup_{i=1}^\ell \overline \Omega_i$. Moreover, given $p_0\in\mathscr R$ there exist $i\neq j$ such that
\[
\lim_{p\to p_0^+} |\nabla_g u_i(p)|^2=\lim_{p\to p_0^-}  |\nabla_g u_j(p)|^2\neq 0,
\]
where $p\to p_0^\pm$ are the limits taken from opposite sides of $\mathscr R$, and for $p_0\in\mathscr S$ we have
\[
\lim_{p\to p_0}|\nabla_g u_i(p)|^2=0\quad \text{for every } i=1,\ldots, \ell.
\]
\item[$(v)$] If $\ell=2$, then\, $u_{\infty,1}-u_{\infty,2}$\,  is a least energy $G$-invariant sign-changing solution to the Yamabe equation \eqref{eq:y}.
\end{itemize}
\end{theorem}

In the nonsymmetric case, the existence of a solution to the system \eqref{eq:s} was recently shown in \cite[Theorem 1.2]{cpt} under the additional assumptions that $(M,g)$ has either dimension $3$ and is not conformal to the standard $3$-sphere, or $(M,g)$ is not locally conformally flat and $\dim M\geq 9$. A nonsymmetric version of Theorem \ref{thm:op} was also established in that work, provided $(M,g)$ is not locally conformally flat, $\dim M\geq 10$, and satisfies a further geometric condition whenever $\dim M=10$; see \cite[Theorem 1.2]{cpt}. 

The existence of a least energy sign-changing solution to the Yamabe equation \eqref{eq:y} was established by Ammann and Humbert in \cite{ah} for manifolds of dimension $\geq 11$ that are not locally conformally flat, without any symmetry assumption, whereas existence of infinitely many $G$-invariant sign-changing solutions to this equation was shown by Clapp and Fernández \cite{cf} for $G$ as above.

An immediate consequence of Theorems \ref{thm:existence} and \ref{thm:op} is the following result.

\begin{theorem}
If $1\leq\dim (Gp)<m$ for every $p\in M$, then for every $\ell\geq 2$ there exists an optimal $(G,\ell)$-partition $\{\Omega_1,\ldots,\Omega_\ell\}$ for the Yamabe equation on $(M,g)$ with the properties stated in items $(iii)$ and $(iv)$ of Theorem \ref{thm:op}.
\end{theorem}

When $M$ is the standard sphere $\sm$ and $G$ is suitably chosen, the optimal partition admits a more accurate description as shown by Clapp and Pistoia \cite{cp} and Clapp,   Saldaña and Szulkin \cite{css} . We shall give an account of known results for the standard sphere in Section \ref{sec:sphere} of this paper. Section \ref{sec:existence} is devoted to the proof of Theorem \ref{thm:existence} and Section \ref{sec:op} to that of Theorem \ref{thm:op}.

\section{Existence of multiple solutions}
\label{sec:existence}

Let $H^1_g(M)$ be the Sobolev space of square integrable functions on $M$ having square integrable first weak derivatives. We write $\langle\,\cdot\,,\,\cdot\,\rangle$ for the Riemannian metric in $(M,g)$ and denote the induced norm by $|\,\cdot\,|$. For $u,v\in H^1_g(M)$ let
$$\langle u,v\rangle_g:=\im\left(\langle\nabla_g u,\nabla_g v\rangle + \kappa_m S_g uv\right)\dm \qquad\text{and}\qquad \|u\|_g:=\sqrt{\langle u,u\rangle_g},$$
where $\nabla_g$ is the weak gradient. Since we are assuming that the conformal Laplacian $\mathscr{L}_g$ is coercive, $\|\cdot\|_g$ is a norm in $H^1_g(M)$, equivalent to the standard one.

Let $G$ be a closed subgroup of the group of isometries of $M$ and set
$$H^1_g(M)^G:=\{u\in H^1_g(M):u\text{ is }G\text{-invariant}\}.$$
If $\dim(Gp)<m$ for $p\in M$ then, for any given $k\in\n$, we may choose $u_1,\ldots u_k\in H^1_g(M)^G$ having pairwise disjoint supports. Hence, $H^1_g(M)^G$ has infinite dimension. The following result will play a crucial role in the proof of Theorem \ref{thm:existence}.

\begin{lemma} \label{lem:hv}
If $\dim(Gp)\geq 1$ for every $p\in M$, then the embedding $H^1_g(M)^G\hookrightarrow L_g^{2^*}(M)$ is compact.
\end{lemma}

\begin{proof}
See \cite[Corollary 1]{hv}.
\end{proof}

We assume from now on that $1\leq\dim(Gp)< m$ for every $p\in M$. Set $\cH^G:=(H^1_g(M)^G)^\ell$ and let $\cJ:\cH^G\to\r$ be given by 
\begin{align} \label{eq:J}
\cJ(u_1,\ldots,u_\ell) :=& \frac{1}{2}\sum_{i=1}^\ell\|u_i\|_g^2 - \frac{1}{2^*}\sum_{i=1}^\ell\im|u_i|^{2^*}\dm - \frac{1}{2^*}\mathop{\sum_{i,j=1}^\ell}_{j\neq i}\im\lambda_{ij}|u_j|^\frac{2^*}{2}|u_i|^\frac{2^*}{2}\dm.
\end{align}
This functional is of class $\cC^1$ and its partial derivatives are
\begin{align*}
\partial_i\cJ(u_1,\ldots,u_\ell)v=&\,\langle u_i,v\rangle_g - \im|u_i|^{2^*-2}u_iv \dm- \mathop{\sum_{j=1}^\ell}_{j\neq i}\im\lambda_{ij}|u_j|^\frac{2^*}{2}|u_i|^{\frac{2^*}{2}-2}u_iv\dm.
\end{align*}
So, by the principle of symmetric criticality \cite{p}, the critical points of $\mathcal{J}:\cH^G\to\r$ are the $G$-invariant solutions to the system
$$\mathscr{L}_g u_i=|u_i|^{2^*-2}u_i + \sum\limits_{\substack{j=1 \\ j\neq i}}^\ell\lambda_{ij}|u_j|^\frac{2^*}{2}|u_i|^\frac{2^*}{2}u_i\quad\text{on }M,\qquad i=1,\ldots,\ell.$$
We are interested in solutions $(u_1,\ldots,u_\ell)$ such that every $u_i$ is nontrivial. They belong to the set
\begin{equation} \label{eq:N}
\cN^G := \{(u_1,\ldots,u_\ell)\in\cH^G:u_i\neq 0, \;\partial_i\cJ(u_1,\ldots,u_\ell)u_i=0, \; \forall i=1,\ldots,\ell\}.
\end{equation}
Note that $\cJ(u_1,\ldots,u_\ell)=\frac{1}{m}\sum_{i=1}^\ell\|u_i\|_g^2$ \ if $(u_1,\ldots,u_\ell)\in \cN^G$.

\begin{lemma} \label{lem:nehari}
There exists $d_0>0$ such that $\|u_i\|_g^2\geq d_0$ for every $(u_1,\ldots,u_\ell)\in\cN^G$ and $i=1,\ldots,\ell$. Therefore, $\cN^G$ is a closed subset of $\cH^G$ and \ $\inf_{\cN^G}\cJ>0$.
\end{lemma}

\begin{proof}
As $\lambda_{ij}<0$, the Sobolev embedding yields a positive constant $C$ such that
\begin{align*}
\|u_i\|_g^2\leq \im|u_i|^{2^*}\dm\leq C\|u_i\|_g^{2^*}\quad \text{ for every \ }(u_1,\ldots,u_\ell)\in \cN^G, \ i=1,\ldots,\ell,
\end{align*}
and our claims follow.
\end{proof}
A solution $\bar u\in\cN^G$ to the system \eqref{eq:s} satisfying $\cJ(\bar u)=\inf_{\cN^G}\cJ$ is called a \emph{least energy $G$-invariant solution}.

The variational approach introduced in \cite{cs} can be immediately adapted to establish the existence of infinitely many fully nontrivial critical points of $\cJ$. We sketch this procedure.

Given $\bar{u}=(u_1,\ldots,u_\ell)$ and $\bar{s}=(s_1,\ldots,s_\ell)\in(0,\infty)^\ell$, we write $\bar{s}\bar{u}:= (s_1u_1,\ldots,s_\ell u_\ell)$. Let $\mathcal S^G:=\{u\in H_g^1(M)^G:\|u\|_g=1\}$, \ $\cT^G:=(\mathcal S^G)^\ell$, \ and
$$\cU^G:=\{\bar{u}\in\cT^G:\bar{s}\bar{u}\in\cN^G\text{ \ for some \ }\bar s\in(0,\infty)^\ell\}.$$

\begin{lemma} \label{lem:U}
\begin{itemize}
\item[$(i)$] Let $\bar u\in\cT^G$. If there exists $\bar s_{\bar u}\in(0,\infty)^\ell$ such that $\bar s_{\bar u}\bar u\in\cN^G$, then $\bar s_{\bar u}$ is unique and satisfies
$$\cJ(\bar s_{\bar u}\bar u)=\max_{\bar s\in(0,\infty)^\ell}\cJ(\bar s\bar u).$$
\item[$(ii)$] $\cU^G$ is a nonempty open subset of the Hilbert manifold $\cT^G$.
\item[$(iii)$] The map $\cU^G\to \cN^G$ given by $\bar u\mapsto\bar s_{\bar u}\bar u$ is a homeomorphism.
\end{itemize}
\end{lemma}

\begin{proof}
The proof is exactly the same as that of \cite[Proposition 3.1]{cs}.
\end{proof}

Define $\Psi:\cU^G\to\r$ by
\begin{equation*}
\Psi(\bar u): = \cJ(\bar s_{\bar u}\bar u).
\end{equation*}
If $\Psi$ is of class $\cC^1$ the norm of $\Psi'(\bar u)$ in the cotangent space $\mathrm{T}_{\bar u}^*(\cT^G)$ to $\cT^G$ at $\bar u$ is defined as
$$\|\Psi'(\bar u)\|_*:=\sup\limits_{\substack{\bar v\in\mathrm{T}_{\bar u}(\cU^G) \\\bar v\neq 0}}\frac{|\Psi'(\bar u)\bar v|}{\|\bar v\|_g},$$
where $\mathrm{T}_{\bar u}(\cU^G)$ is the tangent space to $\cU^G$ at $\bar u$.

A sequence $(\bar u_n)$ in $\cU^G$ is called a $(PS)_c^G$\emph{-sequence for} $\Psi$ if $\Psi(\bar u_n)\to c$ and $\|\Psi'(\bar u_n)\|_*\to 0$, and $\Psi$ is said to satisfy the $(PS)_c^G$\emph{-condition} if every such sequence has a convergent subsequence. Similarly, a $(PS)_c^G$\emph{-sequence for} $\cJ$ is a sequence $(\bar u_n)$ in $\cH^G$ such that $\cJ(\bar u_n)\to 0$ and $\|\cJ'(\bar u_n)\|_{(\cH^G)'}\to 0$. $\cJ$ satisfies the $(PS)_c^G$\emph{-condition} if any such sequence has a convergent subsequence. As usual, $(\cH^G)'$ stands for the dual space of $\cH^G$.

\begin{lemma} \label{lem:psi}
\begin{itemize}
\item[$(i)$] $\Psi\in\cC^1(\cU^G,\r)$,
\begin{equation*}
\Psi'(\bar u)\bar v = \cJ'(\bar s_{\bar u}\bar u)[\bar s_{\bar u}\bar v] \quad \text{for all } \bar u\in\cU^G \text{ and }\bar v\in \mathrm{T}_{\bar u}(\cU^G),
\end{equation*}
and there exists $d_0>0$ such that
$$d_0\,\|\cJ'(\bar s_{\bar u}\bar u)\|_{(\cH^G)'}\leq\|\Psi'(\bar u)\|_*\leq |\bar s_{\bar u}|_\infty\|\cJ'(\bar s_{\bar u}\bar u)\|_{(\cH^G)'}\quad \text{for all } \bar u\in\cU^G,$$
where $|\bar s|_\infty=\max\{|s_1|,\ldots,|s_q|\}$ if $\bar s=(s_1,\ldots,s_q)$.
\item[$(ii)$] Let $\bar u_n\in\cU^G$. If $(\bar u_n)$ is a $(PS)_c^G$-sequence for $\Psi$, then $(\bar s_{\bar u_n}\bar u_n)$ is a $(PS)_c^G$-sequence for $\cJ$.
\item[$(iii)$] Let $\bar u\in\cU^G$. Then, $\bar u$ is a critical point of $\Psi$ if and only if $\bar s_{\bar u}\bar u$ is a critical point of $\cJ$ if and only if $\bar s_{\bar u}\bar u$ is a $G$-invariant solution of \eqref{eq:s}.
\item[$(iv)$] If $(\bar u_n)$ is a sequence in $\cU^G$ and $\bar u_n\to\bar u\in\partial(\cU^G)$, then $\Psi(\bar u_n)\to\infty$.
\item[$(v)$]$\bar{u}\in\cU^G$ if and only if $-\bar{u}\in\cU^G$, and $\Psi(\bar u)=\Psi(-\bar u)$.
\end{itemize}
\end{lemma}

\begin{proof}
The proof of these statements is exactly the same as that of \cite[Theorem 3.3]{cs}.
\end{proof}

\begin{lemma}
$\Psi$ satisfies the $(PS)_c^G$-condition for every $c\in\r$.
\end{lemma}

\begin{proof}
Let $(\bar{v}_n)$ be a $(PS)_c^G$-sequence for $\mathcal{J}$ with $\bar v_n\in\cN^G$. A standard argument shows that this sequence is bounded. Then, using Lemma \ref{lem:hv} as in \cite[Proposition 3.6]{cp}, one sees that $(\bar{v}_n)$ contains a convergent subsequence. The claim now follows from Lemmas \ref{lem:psi}$(ii)$ and \ref{lem:U}$(iii)$.
\end{proof}

Let $\mathcal{Z}$ be a nonempty subset of $\cT^G$ such that $\bar{u}\in\mathcal{Z}$ if and only if $-\bar{u}\in\mathcal{Z}$. Recall that the \emph{genus of $\mathcal{Z}$}, denoted $\mathrm{genus}(\mathcal{Z})$, is the smallest integer $k\geq 1$ such that there exists an odd continuous function $\mathcal{Z}\rightarrow\mathbb{S}^{k-1}$ into the unit sphere $\mathbb{S}^{k-1}$ in $\r^k$. If no such $k$ exists, we define $\mathrm{genus}(\mathcal{Z}):=\infty$. We set $\mathrm{genus}(\emptyset):=0$.

\begin{lemma}
$\mathrm{genus}(\cU^G)=\infty$.
\end{lemma}

\begin{proof}
This is shown following the argument in \cite[Lemma 4.5]{cs}.
\end{proof}

\begin{proof}[Proof of Theorem \ref{thm:existence}]
Lemma \ref{lem:psi}$(iv)$ implies that $\cU^G$ is positively invariant under the negative pseudogradient flow of $\Psi$, so the usual deformation lemma holds true for $\Psi$ in $\cU^G$, see e.g. \cite[Section II.3]{struwe} or \cite[Section 5.3]{w}. As $\Psi$ satisfies the $(PS)_c^G$-condition for every $c\in\r$,
$$\inf_{\cU^G}\Psi=\inf_{\cN^G}\cJ$$
is attained, i.e., the system \eqref{eq:s} has a least energy $G$-invariant solution. Moreover, since $\Psi$ is even and $\mathrm{genus}(\cU^G)=\infty$, a standard variational argument shows that $\Psi$ has an unbounded sequence of critical values.
\end{proof}

\section{The limit profile of minimizers}
\label{sec:op}

We assume throughout that $1\leq\dim(Gp)<m$ for every $p\in M$. Let $\Omega$ be an open $G$-invariant subset of $M$, $H^1_{g,0}(\Omega)$ be the closure of $\cC_c^\infty(\Omega)$ in $H^1_g(M)$ and $H^1_{g,0}(\Omega)^G$ be the space of $G$-invariant functions in $H^1_{g,0}(\Omega)$. The solutions of \eqref{eq:u} are the critical points of the $\cC^2$-functional $J_\Omega:H^1_{g,0}(\Omega)^G\to\r$ given by
$$J_\Omega(u):=\frac{1}{2}\|u\|^2_g -\frac{1}{2^*}\im|u|^{2^*}\dm.$$
The nontrivial ones belong to the Nehari manifold
$$\cN_\Omega^G:=\{u\in H^1_{g,0}(\Omega)^G:u\neq 0\text{ and }J_\Omega'(u)u=0\},$$
which is a natural constraint for $J_\Omega$. So, a minimizer for $J_\Omega$ over $\cN_\Omega^G$ is a nontrivial solution of \eqref{eq:u}. We call it a \emph{least energy $G$-invariant solution}. A standard argument using Lemma \ref{lem:hv} shows that a minimizer does exist. Note that $J_\Omega(u)=\frac{1}{m}\|u\|_g^2=\frac{1}{m}\im|u|^{2^*}\dm$ if $u\in\cN_\Omega^G$, so the quantity defined in \eqref{eq:cOmega} is
$$c_\Omega^G=\inf_{u\in\cN_\Omega^G}J_\Omega(u).$$

\begin{proof}[Proof of Theorem \ref{thm:op}]
Let $\lambda_n\to-\infty$. We write $\cJ_n$ and $\cN_n^G$ for the functional and the set defined in \eqref{eq:J} and \eqref{eq:N} with $\lambda_{ij}=\lambda_n$ for all $i\neq j$.
Let
\begin{align*}
\cM_\ell^G:=&\{(v_1,\ldots,v_\ell)\in\cH^G:v_i\neq 0, \ \|v_i\|_g^2=\im|v_i|^{2^*}\dm, \ v_iv_j=0\text{ on }M\text{ if }i\neq j\},\\
c_\ell^G:=&\inf_{(v_1,\ldots,v_\ell)\in\cM_\ell}\,\frac{1}{m}\sum_{i=1}^\ell\|v_i\|_g^2.
\end{align*}
Since $\dim(Gp)<m$ for $p\in M$, the set $\cM_\ell^G$ is nonempty and, so, $c_\ell^G<\infty$. 

Let $\bar u_n=(u_{n,1},\ldots,u_{n,\ell})\in\cN_n^G$ be such that $\cJ_n(\bar u_n)=\inf_{\cN_n^G}\cJ_n$ and $u_{n,i}>0$ for every $n$ and $i$. Noting that $\cM_\ell^G\subset\cN_n^G$ for every $n\in\n$ and recalling Lemma \ref{lem:nehari}, we see that
$$0<\inf_{\cN_n^G}\cJ_n=\frac{1}{m}\sum_{i=1}^\ell\|u_{n,i}\|_g^2\leq c_\ell^G<\infty\qquad\text{for every \ }n\in\n.$$
Applying Lemma \ref{lem:hv} and passing to a subsequence, we get that $u_{n,i} \rh u_{\infty,i}$ weakly in $H^1_g(M)^G$, $u_{n,i} \to u_{\infty,i}$ strongly in $L^{2^*}_g(M)$ and $u_{n,i} \to u_{\infty,i}$ a.e. on $M$, for each $i=1,\ldots,\ell$. Hence, $u_{\infty,i} \geq 0$. 

As $\bar u_n\in\cN_n^G$, we have that
\begin{align*}
0&\leq\im|u_{n,j}|^{\frac{2^*}{2}}|u_{n,i}|^{\frac{2^*}{2}}\dm\leq \frac{\im|u_{n,i}|^{2^*}\dm}{-\lambda_n}\leq \frac{C}{-\lambda_n}\qquad\text{for each pair \ }j\neq i,
\end{align*}
and, letting $n\to\infty$, we obtain
$$\im |u_{\infty,j}|^{\frac{2^*}{2}}|u_{\infty,i}|^{\frac{2^*}{2}}\dm = 0.$$
So $u_{\infty,j} u_{\infty,i} = 0$ a.e. on $M$ whenever $i\neq j$. We also have that
\begin{equation*}
0<d_0\leq\|u_{n,i}\|_g^2 \leq \im|u_{n,i}|^{2^*}\dm\qquad\text{for every \ } n\in\n\text{ \ and \ }i=1,\ldots,\ell,
\end{equation*}
with $d_0$ as in Lemma \ref{lem:nehari}. Passing to the limit as $n\to\infty$, we see that $u_{\infty,i}\neq 0$ and
\begin{equation*}
\|u_{\infty,i}\|_g^2 \leq \im|u_{\infty,i}|^{2^*}\dm\qquad\text{for every \ }i=1,\ldots,\ell.
\end{equation*}
Hence, there exists $t_i\in(0,1]$ such that $\|t_iu_{\infty,i}\|_g^2 = \im|t_iu_{\infty,i}|^{2^*}\dm$. Consequently, $(t_1u_{\infty,1},\ldots,t_\ell u_{\infty,\ell})\in \cM_\ell^G$ and
$$c_\ell^G\leq \frac{1}{m}\sum_{i=1}^\ell\|t_iu_{\infty,i}\|_g^2\leq \frac{1}{m}\sum_{i=1}^\ell\|u_{\infty,i}\|_g^2\leq \frac{1}{m}\liminf_{n\to\infty}\sum_{i=1}^\ell\|u_{n,i}\|_g^2 \leq c_\ell^G.$$
It follows that $t_i=1$ and $u_{n,i} \to u_{\infty,i}$ strongly in $H^1_g(M)$. But then, $\|u_{\infty,i}\|_g^2 = \im|u_{\infty,i}|^{2^*}\dm$ and, passing to the limit in
\begin{align*}
\sum_{i=1}^\ell\|u_{n,i}\|_g^2 &=\sum_{i=1}^\ell |u_{n,i}|_{g,2^*}^{2^*}+\mathop{\sum_{i,j=1}^\ell}_{j\neq i} \int_M \lambda_{n}|u_{n,j}|^\frac{2^*}{2}|u_{n,i}|^\frac{2^*}{2},
\end{align*}
we obtain 
$$\lim_{n\to\infty}\int_M \lambda_n|u_{\infty,j}|^\frac{2^*}{2}|u_{\infty,i}|^\frac{2^*}{2}=0\qquad\text{for every \ }i\neq j.$$
Moreover, $(u_{\infty,1},\ldots,u_{\infty,\ell})\in \cM_\ell^G$ and 
$$\frac{1}{m}\sum_{i=1}^\ell\|u_{\infty,i}\|_g^2=c_\ell^G.$$

We have now all assumptions needed to apply \cite[Lemmas 4.3 and 4.4]{cpt} and conclude that $(u_{n,i})$ is uniformly bounded in the Hölder norm, i.e., for any $\alpha\in(0,1)$ there exists $C_\alpha>0$ such that
$$\|u_{n,i}\|_{\cC^{0,\alpha}(M)}\leq C_\alpha\qquad\text{for every \ } n\in\n\text{ \ and \ }i=1,\ldots,\ell.$$
As a consequence, $u_{n,i}\to u_{\infty,i}$ in $\cC^{0,\alpha}(M)$ for every $i$. In particular, $u_{\infty,i}$ is continuous and $G$-invariant, so the set $\o_i:=\{p\in M:u_{\infty,i}(p)>0\}$ is open and $G$-invariant. Since $u_{\infty,i}u_{\infty,j}=0$ if $i\neq j$, we have that $\o_i\cap\o_j=\emptyset$. This shows that $\{\o_1,\ldots,\o_\ell\}\in\cP_\ell^G$.

We claim that $u_{\infty,i}$ is a least energy $G$-invariant solution to \eqref{eq:u} in $\o_i$ for all $i$. Otherwise, $J_{\Omega_i}(u_{\infty,i})>c_{\Omega_i}^G$ for some $i$ and there would exist $v_i\in\cN_{\Omega_i}^G$ with $c_{\Omega_i}^G<J_{\Omega_i}(v_i)<J_{\Omega_i}(u_{\infty,i})$. But then, setting $v_j:=u_{\infty,j}$ for $j\neq i$, we would have that $(v_1,\ldots,v_\ell)\in\cM_\ell^G$ and
$$\frac{1}{m}\sum_{i=1}^\ell\|v_i\|_g^2<\frac{1}{m}\sum_{i=1}^\ell\|u_{\infty,i}\|_g^2=c_\ell^G,$$
contradicting the definition of $c_\ell^G$. Hence, $J_{\Omega_i}(u_{\infty,i})=c_{\Omega_i}^G$ for all $i=1,\ldots,\ell$, as claimed. A similar argument shows that $\o_i$ is connected.

As a consequence, if $\{\Theta_1,\ldots,\Theta_\ell\}\in\cP_\ell^G$, taking $w_i\in\cN_{\Theta_i}^G$ with $J_{\Theta_i}(w_i)=c_{\Theta_i}^G$, we have that $(w_1,\ldots,w_\ell)\in\cM_\ell^G$ and, therefore,
$$\sum_{i=1}^\ell c_{\o_i}^G=\frac{1}{m}\sum_{i=1}^\ell\|u_{\infty,i}\|_g^2=c_\ell^G\leq\frac{1}{m}\sum_{i=1}^\ell\|w_i\|_g^2=\sum_{i=1}^\ell c_{\Theta_i}^G.$$
This shows that $\{\o_1,\ldots,\o_\ell\}$ is an optimal $(G,\ell)$-partition, and completes the proof of statements $(i)$ and $(iii)$.

Statements $(ii)$ and $(iv)$ are local. In local coordinates the system \eqref{eq:s} with $\lambda_{ij}$ replaced by $\lambda_n$ becomes
\[
-\mathrm{div}(A(x)\nabla v_i)=f_i(x,v_i) + a(x) \sum_{\substack{j=1 \\ j\neq i}}^\ell \lambda_n |v_j|^\frac{2^*}{2} |v_i|^{\frac{2^*}{2}-2}v_i,\qquad x\in\Omega,
\]
where $\Omega\subset\rm$ is open and bounded, $a(x):=\sqrt{|g(x)|}$, $A(x):=\sqrt{|g(x)|}(g^{kl}(x))$, $f_i(x,s):=a(x)(|s|^{2^*-2}s-\kappa_mS_g(x)s)$ and, as usual, $|g|$ denotes the determinant of the metric $g=(g_{kl})$ in local coordinates and $(g^{kl})$ its inverse. Applying \cite[Theorem C.1]{cpt} we conclude that $(ii)$ and $(iv)$ are true locally on $M$, hence also globally.

A standard argument yields the proof of statement $(v)$. Namely, the $G$-invariant sign-changing solutions to the Yamabe equation \eqref{eq:y} belong to the set
$$\cE_M^G:=\{u\in\cN_M^G:u^+\in\cN_M^G\text{ \ and \ }u^-\in\cN_M^G\},$$
where $u^+:=\max\{u,0\}\neq 0$ and $u^-:=\min\{u,0\}\neq 0$. Moreover, as shown in \cite[Lemma 2.6]{ccn}, any minimizer of $J_M$ on $\cE_M^G$ is a $G$-invariant sign-changing solution of \eqref{eq:y}. For every $u\in\cE_M^G$, we have that $(u^+,u^-)\in\cM_2^G$ and $J_M(u)=\frac{1}{m}(\|u^+\|_g^2+\|u^-\|_g^2)$. Therefore,
$$\inf_{\cE_M^G}J_M\geq c_2^G=\frac{1}{m}(\|u_{\infty,1}\|_g^2+\|u_{\infty,2}\|_g^2).$$
As $u_{\infty,1}-u_{\infty,2}\in\cE_M^G$, it is a minimizer of $J_M$ on $\cE_M^G$. This completes the proof.
\end{proof}

\section{The Yamabe system on the standard sphere}
\label{sec:sphere}

In this section we give an account of some known results for the system \eqref{eq:s} on the standard sphere.

The stereographic projection $\sigma:\sm\smallsetminus\{N\}\to\rm$ from the north pole $N$ is a
conformal diffeomorphism. The conformal invariance of the operator $\mathscr{L}_{\bar g}$ (see \cite[Proposition 6.1.1]{h}) allows to establish a one-to-one correspondence between solutions to the system \eqref{eq:s} on the standard sphere $\sm$ and solutions to the system
\begin{equation} \label{eq:rm}
\begin{cases}
-\Delta v_i = |v_i|^{2^*-2}v_i + \sum\limits_{\substack{j=1 \\ j\neq i}}^\ell\lambda_{ij}|v_j|^\frac{2^*}{2}|v_i|^\frac{2^*}{2}v_i \\
v_i\in D^{1,2}(\rm), \quad v_i\neq 0,\qquad i=1,\ldots,\ell,
\end{cases}
\end{equation}
where $D^{1,2}(\rm):=\{v\in L^{2^*}(\rm):\nabla v\in L^2(\rm,\rm)\}$; see \cite{cp,css}.

Let us first consider the case of the single equation
\begin{equation} \label{eq:srm}
-\Delta u= |u|^{2^*-2}u,\qquad u\in D^{1,2}(\rm).
\end{equation}
It is well known \cite{au,oba,tal} that all positive solutions to \eqref{eq:srm} are the so-called {\em standard bubbles}
$$U_{\delta,y}(x):=\mathfrak c_m\frac{\delta^\frac{m-2}{2}}{(\delta^2+|x-y|^2)^\frac{m-2}{2}},\qquad x,y\in\mathbb R^m,\ \delta>0.$$
The existence of sign-changing solutions was first established by W.Y. Ding in \cite{d}. He considered solutions that are invariant under the action of the group $\Gamma:=O(n_1)\times O(n_2)$ with $n_1,n_2\geq 2$ and $n_1+n_2=m+1$, acting on $\sm\subset\r^{m+1}\equiv\r^{n_1}\times\r^{n_2}$ in the obvious way. Using a variational approach he proved the existence of infinitely many $\Gamma$-invariant sign-changing solutions to the Yamabe equation \eqref{eq:y} on $\sm$ or, equivalently, to the equation \eqref{eq:srm} on $\mathbb R^m$.
  
In \cite{dmpp1,dmpp2}, del Pino, Musso, Pacard and Pistoia exploited the symmetries of the sphere to build solutions to \eqref{eq:srm} which have large energy and concentrate along some special submanifolds of $\sm$. In particular, for $m\ge4$ they  obtained sequences of solutions to the Yamabe equation whose energy concentrates along one great circle or finitely many great circles that are linked to each other (corresponding to Hopf links embedded in $\mathbb S^3\times\{0\} \subset\mathbb S^m$), and for $m\ge5$ they also obtained sequences of solutions whose energy concentrates along a two-dimensional Clifford torus in $\mathbb S^3\times\{0\} \subset\mathbb S^m$.  These solutions are built via gluing techniques (e.g., a Ljapunov-Schmidt procedure) and can be described as the superposition of the constant solution to \eqref{eq:y} with a large number of copies of negative solutions of  \eqref{eq:y} which are highly concentrated at points evenly arranged along some special submanifolds of the sphere.

Recently, Fernández and Petean in \cite{fp} established the existence of solutions to the Yamabe problem \eqref{eq:y} on $\sm$ with precisely $\ell$ nodal domains which are isoparametric hypersurfaces, for every $\ell\geq 2$, using ODE techniques.

Regarding the competitive system (where all $\lambda_{ij}<0$), Guo, Li and Wei \cite{glw} established the existence of solutions for \eqref{eq:rm} with $\ell=2$ and $\lambda_{12}<0$ in $\r^3$, and using the approach developed in \cite{dmpp1} they  built a sequence of positive nonradial solutions whose first component looks like the constant function and the second component resembles a large number of copies of positive solutions of  \eqref{eq:y} concentrated at points that are  placed along a circle. The argument of their proof, which relies on the Ljapunov-Schmidt procedure, cannot be extended to higher dimensions because the coupling terms have linear (if $m=4$) or sublinear (if $m\ge5$) growth.

Following the variational approach presented in the previous sections, successively Clapp and Pistoia \cite{cp}, Clapp and Szulkin \cite{cs} and Clapp, Saldaña and Szulkin \cite{css} found $\Gamma$-invariant solutions to the competitive Yamabe system \eqref{eq:s} on $\sm$ for the groups considered by Ding, and described the limit profile of least energy solutions as $\lambda_{ij}\to\-\infty$. In this particular case, Theorems \ref{thm:existence} and \ref{thm:op} were proved in \cite{cp,cs} and \cite{cp,css} respectively. In addition, a more accurate description of the optimal $(\Gamma,\ell)$-partition of $\sm$ is provided in \cite{cp,css}. Note that the $\Gamma$-orbit of a point $p\in\sm$ is diffeomorphic to either $\mathbb{S}^{n_1-1}$, or $\mathbb{S}^{n_1-1}\times\mathbb{S}^{n_2-1}$, or $\mathbb{S}^{n_2-1}$, and that the map $q:\sm\to[0,\pi]$ given by
\begin{equation} \label{eq:orbit_map}
q(x,y):=\arccos(|x|^2-|y|^2),\quad \text{where \ }x\in\r^{n_1}, \ y\in\r^{n_2},
\end{equation}
is a quotient map identifying each $\Gamma$-orbit in $\sm$ to a single point. So $q$ maps a $(\Gamma,\ell)$-partition of $\sm$ onto a partition of $[0,\pi]$ by relatively open subintervals. This last partition can be ordered. Taking advantage of this fact, the following result was proved in \cite{cp,css}.

\begin{theorem} \label{thm:opsm}
Let $\Gamma:=O(n_1)\times O(n_2)$ with $n_1,n_2\geq 2$ and $n_1+n_2=m+1$ and let $M:=\sm$. Then, the optimal $(\Gamma,\ell)$-partition $\{\o_1,\ldots,\o_\ell\}\in\cP_\ell^\Gamma$ of $\sm$ given by Theorem \ref{thm:op} has the following properties: $\o_1,\ldots,\o_\ell$ are smooth and connected, $\overline{\o_1\cup\cdots\cup \o_\ell}=\sm$ and, after reordering,
\begin{itemize}
\item[•] $\o_1\cong\mathbb{S}^{n_1-1}\times \mathbb{B}^{n_2}$,\quad $\o_i\cong\mathbb{S}^{n_1-1}\times\mathbb{S}^{n_2-1}\times(0,1)$ if $i=2,\ldots,\ell-1$, \quad and \quad $\o_\ell\cong\mathbb{B}^n_1\times \mathbb{S}^{n_2-1}$,
\item[•] $\overline{\o}_i\cap \overline{\o}_{i+1}\cong\mathbb{S}^{n_1-1}\times\mathbb{S}^{n_2-1}$ \quad and\quad $\overline{\o}_i\cap \overline{\o}_j=\emptyset$\, if\, $|j-i|\geq 2$,
\item[•] the function
$$u:=\sum_{i=1}^\ell(-1)^{i-1}u_{\infty,i}$$
is a $\Gamma$-invariant sign-changing solution to the Yamabe problem \eqref{eq:y} on $\sm$ with precisely $\ell$ nodal domains, and $u$ has least energy among all such solutions.
\end{itemize}
\end{theorem}

As shown by Theorems \ref{thm:op} and \ref{thm:opsm} the competitive Yamabe system \eqref{eq:s} gives rise to one or even multiple solutions to the Yamabe equation \eqref{eq:y}. The opposite question has been also considered in the literature. Let us call a solution $\mathbf{u} = (u_1,\ldots,u_\ell)$ to the Yamabe system \eqref{eq:rm} \emph{fully synchronized} if there exist $c_i\neq 0$ and a nontrivial solution $u$ to the single equation \eqref{eq:srm} such that $u_i=c_i u$ for all $i=1,\ldots,\ell$. It is readily seen that $(c_1u,\ldots,c_\ell u)$ solves \eqref{eq:rm} for any solution $u$ of \eqref{eq:srm} iff 
$\mathbf{c} =(c_1,\ldots,c_\ell)\in\r^\ell$ solves the algebraic system
\begin{equation}\label{ss1}
c_i=|c_i |^{2^*- 2}c_i+\sum_{\substack{j=1\\ j\not=1}}^\ell \lambda_{ij} |c_j|^\frac{2^*}{2}|c_i|^{\frac{2^*}{2}-2}c_i,\quad c_i>0,\quad\text{for every \ }i=1,\dots,\ell.
\end{equation}
There are several results concerning the solvability of \eqref{ss1}. The easiest case is when $m=4$ (i.e., $2^*=4$) and $\ell=2$. Indeed, a straightforward computation shows that a solution to \eqref{ss1} exists if and only if
 $\lambda_{12}>-1$ and $\lambda_{12}\not=1$. Bartsch proved in \cite[Proposition 2.1]{b} that, if $2^*=4$ and $\ell\ge2$, a fully synchronized solution to \eqref{ss1} exists when  $\lambda_{ij}:=\lambda\not=1$ for all $i\not=j$ and $\lambda>\overline\lambda$ for some $\overline\lambda<0 $ , while Chen and Zou \cite[Theorem 1.1]{cz2} showed that  if $m\ge5$ (i.e., $2^*<4$) and $\ell=2$ a fully synchronized solution to \eqref{ss1} always exists provided $\lambda_{12}>0$. Recently, Clapp and Pistoia complemented these results in \cite{cp1} showing that, if the system is purely cooperative (i.e., $\lambda_{ij}\geq 0$ for all $i,j=1,\ldots,\ell$, $i\neq j$), there exists a solution to \eqref{ss1}. We recall that in the purely cooperative case every positive solution of \eqref{eq:rm}  with $\ell=2$ is fully syncronized, as shown by Guo and Liu in \cite{gl}. On the other hand, it is shown in \cite{cp,cs} that there exists $\lambda^*<0$ such that the system \eqref{eq:rm} does not have a fully synchronized solution if $\lambda_{ij}<\lambda^*$ for all pairs $i\neq j$.

System \eqref{eq:rm} with mixed couplings (i.e., $\lambda_{ij}$ can be positive or negative) has been recently studied by Clapp and Pistoia \cite{cp1}. Let $\lambda_{ii}=1$ and assume  the matrix $(\lambda_{ij})$ is symmetric and admits a block decomposition as follows: For some $1 < q < \ell$ there exist \ $0=\ell_0<\ell_1<\dots<\ell_{q-1}<\ell_q=\ell$ \ such that, if we set
\begin{align*}
 &I_h:= \{i \in  \{1,\dots,\ell\}:  \ell_{h-1} < i \le \ell_h \},\\
 &\cI_h:=I_h\times I_h,\qquad \cK_h:=\big\{(i,j)\in I_h\times I_k: k\in\{1,\ldots,q\}\smallsetminus\{h\}\big\},
\end{align*}
then 
\begin{equation*}
\lambda_{ij}>0\ \text{ if }\ (i,j)\in \cI_h \quad \text{ and }\quad \lambda_{ij}\le0\ \text{if}\ (i,j)\in \cK_h,\quad h=1,\ldots,q.
\end{equation*}
According to the above decomposition, we shall write a solution $\bf u=(u_1,\ldots,u_\ell)$ to \eqref{eq:rm} in block-form as
$$\bf u=(\bar u_1,\ldots,\bar u_q)\qquad\text{with \ }\bar u_h=(u_{\ell_{h-1}+1},\ldots,u_{\ell_h}).$$
$\bf u$ is called \emph{fully nontrivial} if every component $u_i$ is different from zero.  In \cite{cp1} it is proved that if, either $m\geq 5$, or $m=4$ and $\lambda_{ij}=:b_h>1$ for all $i,j\in I_h$ with $i\neq j$, the system \eqref{eq:rm} has a fully nontrivial solution if $\max_{(i,j)\in\cK_h}|\lambda_{ij}|<\Lambda$ for some $\Lambda>0$. This solution is invariant under the conformal action on $\rm$ of the group $\Gamma$ defined above.

Finally, we would like to mention a couple of results, one of them by Grossi, Gladiali and Troestler \cite{ggt} where they  give sufficient conditions on the matrix $(\lambda_{ij})$ to ensure the existence of solutions bifurcating from the bubble of the critical Sobolev equation, and another one by Druet and Hebey \cite{dh} where they study the stability of solutions to \eqref{eq:rm} under linear perturbation.

\bigskip

\begin{flushleft}
\textbf{M\'onica Clapp}\\
Instituto de Matemáticas\\
Universidad Nacional Autónoma de México\\
Circuito Exterior, Ciudad Universitaria\\
04510 Coyoacán, Ciudad de México, Mexico\\
\texttt{monica.clapp@im.unam.mx} 
\medskip

\textbf{Angela Pistoia}\\
Dipartimento di Metodi e Modelli Matematici\\
La Sapienza Università di Roma\\
Via Antonio Scarpa 16 \\
00161 Roma, Italy\\
\texttt{angela.pistoia@uniroma1.it} 
\end{flushleft}


\begin{thebibliography}{99}

\bibitem{ah}Ammann, Bernd; Humbert, Emmanuel: The second Yamabe invariant. J. Funct. Anal. 235 (2006), no. 2, 377--412.

\bibitem{b}Bartsch, Thomas: Bifurcation in a multicomponent system of nonlinear Schrödinger equations. J. Fixed Point Theory Appl. 13 (2013), no. 1, 37–50.

\bibitem{au} Aubin,  Thierry: Problèmes isopérimétriques et espaces de Sobolev, J. Differ. Geom. 11 (4) (1976) 573--598. 

\bibitem{ccn}Castro, Alfonso; Cossio, Jorge; Neuberger, John M.: A sign-changing solution for a superlinear Dirichlet problem. Rocky Mountain J. Math. 27 (1997), no. 4, 1041–1053.

\bibitem{clll} Chang, Shu-Ming; Lin,  Chang-Shou; Lin, Tai-Chia, Lin, Wen-Wei: Segregated nodal domains of two-dimensional multispecies Bose-Einstein condensates. Phys. D 196 (2004), 341--361.

\bibitem{cf} Clapp, Mónica; Fernández, Juan Carlos: Multiplicity of nodal solutions to the Yamabe problem. Calc. Var. Partial Differential Equations 56 (2017), no. 5, Paper No. 145, 22 pp. 

\bibitem{cp}Clapp, Mónica; Pistoia, Angela: Existence and phase separation of entire solutions to a pure critical competitive elliptic system. Calc. Var. Partial Differential Equations 57 (2018), no. 1, Paper No. 23, 20 pp.

\bibitem{cp1}Clapp, Mónica; Pistoia, Angela: Fully nontrivial solutions to elliptic systems with mixed couplings (2021) ArXiv:2106.01637

\bibitem{cpt}Clapp, Mónica; Pistoia, Angela; Tavares, Hugo: Yamabe systems, optimal partitions and nodal solutions to the Yamabe equation. Preprint 2021, arXiv:2106.00579.

\bibitem{css}Clapp, Mónica; Saldaña, Alberto; Szulkin, Andrzej: Phase separation, optimal partitions and nodal solutions to the Yamabe equation on the sphere. International Mathematics Research Notices 2021 (2021), no. 5, 3633--3652.

\bibitem{cs}Clapp, Mónica; Szulkin, Andrzej: A simple variational approach to weakly coupled competitive elliptic systems. Nonlinear Differential Equations and Applications NoDEA 26:26 (2019), 21 pp.

\bibitem{ctv1}Conti, Monica; Terracini, Susanna; Verzini, Gianmaria: Nehari's problem and competing species systems. Ann. Inst. H. Poincar\'e Anal. Non Lin\'eaire 19 (2002), no. 6, 871--888.

\bibitem{ctv2}Conti, Monica; Terracini, Susanna; Verzini, Gianmaria: A variational problem for the spatial segregation of reaction-diffusion systems. Indiana Univ. Math. J. 54 (2005), no. 3, 779--815.

\bibitem {cz2}Chen, Zhijie; Zou, Wenming: Positive least energy solutions and phase separation for coupled Schrödinger equations with critical exponent: higher dimensional case. Calc. Var. Partial Differential Equations 52 (2015), no. 1-2, 423-467.

\bibitem{dmpp1} del Pino, Manuel;  Musso, Monica;  Pacard, Frank; Pistoia, Angela: Large energy entire solutions for the Yamabe equation. J. Differential Equations 251 (9) (2011) 2568--2597.

\bibitem{dmpp2}del Pino, Manuel;  Musso, Monica;  Pacard, Frank; Pistoia, Angela: Torus action on Sn and sign-changing solutions for conformally invariant equations. Ann. Sc. Norm. Super. Pisa Cl. Sci. (5) 12 (1) (2013) 209--237.

\bibitem{d}Ding, Wei Yue: On a conformally invariant elliptic equation on $R^n$. Comm. Math. Phys. 107 (1986), no. 2, 331--335.

 \bibitem{dh} Druet, Olivier; Hebey, Emmanuel: 
 Stability for strongly coupled critical elliptic systems in a fully inhomogeneous medium.  
 Anal. PDE 2 (2009), no. 3, 305--359.
 
\bibitem{fp} Fernández, Juan Carlos; Petean, Jimmy: Low energy nodal solutions to the Yamabe equation. J. Differential Equations 268 (2020), no. 11, 6576--6597.

\bibitem{ggt} Gladiali, Francesca; Grossi, Massimo; Troestler, Christophe:  
A non-variational system involving the critical Sobolev exponent. The radial case. 
J. Anal. Math. 138 (2019), no. 2, 643--671.

\bibitem{glw} Guo, Yuxia; Li, Bo; Wei, Juncheng: Entire nonradial solutions for non-cooperative coupled elliptic system with critical exponents in $\mathbb R^3$. J. Differential Equations 256 (2014), no. 10, 3463--3495.

\bibitem{gl} Guo, Yuxia; Liu, Jiaquan: Liouville type theorems for positive solutions of elliptic system in $\mathbb R^n$. Comm. Partial Differential Equations 33 (2008), no. 1--3, 263–284.

\bibitem{h} Hebey, E.: Introduction à l’analyse non linéaire sur les variétés. Diderot, Paris, 1997.

\bibitem{hv}Hebey, Emmanuel; Vaugon, Michel: Sobolev spaces in the presence of symmetries. J. Math. Pures Appl. (9) 76 (1997), no. 10, 859–881. 

\bibitem{oba}Obata, Morio: The conjectures on conformal transformations of Riemannian manifolds.
J. Differential Geometry 6 (1971/72), 247--258.
 
\bibitem{p}Palais, Richard S.: The principle of symmetric criticality. Comm. Math. Phys. 69 (1979), no. 1, 19–30. 

\bibitem{struwe}Struwe, Michael: Variational methods. Applications to nonlinear partial differential equations and Hamiltonian systems. Second edition. Ergebnisse der Mathematik und ihrer Grenzgebiete, 34. Springer-Verlag, Berlin, 1996.

\bibitem{sttz} Soave, Nicola; Tavares, Hugo; Terracini, Susanna; Zilio, Alessandro: H\"older bounds and regularity of emerging free boundaries for strongly competing Schr\"odinger equations with nontrivial grouping. Nonlinear Anal. 138 (2016), 388--427.

\bibitem{tal}Talenti, Giorgio:
Best constant in Sobolev inequality.
Ann. Mat. Pura Appl. (4) 110 (1976), 353--372.

\bibitem{w}Willem, Michel: Minimax theorems. Progress in Nonlinear Differential Equations and their Applications, 24. Birkhäuser Boston, Inc., Boston, MA, 1996. 
\end{thebibliography}
\end{document}